\documentclass[12pt, reqno]{amsart}
\usepackage[utf8]{inputenc}
\usepackage{amsmath}
\usepackage{amsfonts}
\usepackage{amssymb}
\usepackage{amsthm}
\usepackage{bm}
\usepackage{bbm}
\usepackage{xcolor}
\usepackage{url}
\usepackage{hyperref}
\usepackage{mathtools}

\numberwithin{equation}{section}

\theoremstyle{plain}
\newtheorem{theorem}{Theorem}[section]
\newtheorem{lemma}[theorem]{Lemma}

\newtheorem{proposition}[theorem]{Proposition}
\newtheorem{conjecture}[theorem]{Conjecture}

\title[irrationality of a prime factor series]{The irrationality of a prime factor series under a prime tuples conjecture}
\author{Kyle Pratt}
\address{Brigham Young University, Department of Mathematics, Provo, UT 84602, USA}
\email{kyle.pratt@mathematics.byu.edu}

\subjclass[2010]{}

\allowdisplaybreaks

\begin{document}
\date{}

\maketitle

\begin{abstract}
Let $\omega(n)$ denote the number of distinct prime factors of $n$. Assuming a suitably uniform version of the prime $k$-tuples conjecture, we show that the number
\begin{align*}
\sum_{n=1}^\infty \frac{\omega(n)}{2^n}
\end{align*}
is irrational. This settles (conditionally) a question of Erd\H{o}s.
\end{abstract}

\section{Introduction}

Let $\tau(n),\varphi(n), \sigma(n)$, and $\omega(n)$ denote the number of divisors of $n$, the Euler totient of $n$, the sum of divisors of $n$, and the number of distinct prime factors of $n$, respectively. If $t\geq 2$ is an integer, then Erd\H{o}s \cite{Erdos1948} proved that the series
\begin{align*}
\sum_{n\geq 1} \frac{\tau(n)}{t^n}
\end{align*}
is irrational, and noted that proving analogous results for
\begin{align*}
\sum_{n\geq 1} \frac{\varphi(n)}{t^n}, \ \ \ \ \ \sum_{n\geq 1}\frac{\sigma(n)}{t^n}, \ \ \ \ \ \sum_{n\geq 1}\frac{\omega(n)}{t^n}
\end{align*}
seems ``to present difficulties''\footnote{See also \url{https://www.erdosproblems.com/69}, \url{https://www.erdosproblems.com/249}, and \url{https://www.erdosproblems.com/250}}. Erd\H{o}s repeatedly mentioned the problems of proving the irrationality of these series (see e.g. \cite{Erdos1957,Erdos1988}, \cite[p. 61]{EG1980}). The series involving $\sigma(n)$ is now known to be transcendental as a corollary of deep work of Nesterenko \cite{Nest1996}, but the questions for the other two sums are open. Here we make further, albeit conditional, progress.

Let $\mathcal{L} = \{L_1,\ldots,L_K\}$ be a set of distinct linear forms $L_k(n) = a_kn+b_k$, where the coefficients $a_k,b_k$ are positive integers. For a prime $p$, let $\omega_\mathcal{L}(p)$ denote the number of roots of $\prod_{k=1}^K L_k(n)$ modulo $p$. We say $\mathcal{L}$ is \emph{admissible} if $\omega_\mathcal{L}(p) < p$ for every prime $p$. One has the following conjecture (see, e.g., \cite[p. 3563]{Ford2022}, \cite[p. 174]{Gran2015}).

\begin{conjecture}[Prime $K$-tuples conjecture]\label{conj:basic HL conj}
Let $\mathcal{L} = \{L_1,\ldots,L_K\}$ be an admissible set of linear forms. Then there are infinitely many integers $n$ such that $L_1(n),\ldots,L_K(n)$ are all prime.
\end{conjecture}

For many applications, one desires a quantitative version of Conjecture \ref{conj:basic HL conj}. For instance, given a large $x$, one would like to know how many $n\leq x$ there are such that all of the $L_k(n)$ are prime. Additionally, it is desirable to allow for some uniformity in the conjecture, such that the coefficients of the linear forms and the number of linear forms can change as $x$ increases.

\begin{conjecture}[Quantitative prime $K$-tuples conjecture]\label{conj:unif HL conj}
Let $\mathcal{L} = \{L_1,\ldots,L_K\}$ be an admissible set of linear forms, where $L_k(n) = a_kn+b_k$ with the $a_k,b_k$ positive integers. Define the singular series
\begin{align*}
\mathfrak{S}(\mathcal{L}) = \prod_p \left(1 - \frac{\omega_\mathcal{L}(p)}{p} \right) \left(1 - \frac{1}{p} \right)^{-K}.
\end{align*}
If $x$ is sufficiently large, if $a_k,b_k \leq (\log \log x)^{100}$, and if $K \leq 100\log \log \log x$, then
\begin{align*}
\sum_{\substack{n\leq x \\ L_k(n) \textup{ is prime for }1\leq k \leq K}} \!\!\!\!\!\!\!\!\!\!\!\!\!\!\!\!\! 1 \, \, \, \, \, \, \, \, \, \, \, &=(1+o(1)) \mathfrak{S}(\mathcal{L}) \frac{x}{(\log x)^K},
\end{align*}
where $o(1)$ denotes a quantity which goes to zero as $x$ goes to infinity.
\end{conjecture}

There has been recent progress on versions of Conjecture \ref{conj:unif HL conj}, in which several rather than all of the $L_k(n)$ are prime \cite{May2016}.

The statement of Conjecture \ref{conj:unif HL conj} suffices for our present work, but could be refined further. There are various statements of a uniform prime tuples conjecture in the literature (e.g. \cite[Conjecture 1.3]{Kup2023}), but usually only in the case in which $L_k(n)=n+b_k$ for each $k$, so that the coefficients $a_k$ are all equal to one.

Our main result is that Conjecture \ref{conj:unif HL conj} has implications for one of the irrationality questions of Erd\H{o}s.

\begin{theorem}\label{thm:main theorem}
Assume Conjecture \ref{conj:unif HL conj}. Then, for every integer $t\geq 2$, the number
\begin{align}\label{eq:series we prove is irrational}
\sum_{n\geq 1} \frac{\omega(n)}{t^n}
\end{align}
is irrational.
\end{theorem}
Of course, we recover the result mentioned in the abstract by specializing $t=2$. 

We use the standard asymptotic notation $O(\cdot),o(\cdot), \ll, \gg$. These symbols are often used in the context of a large real $x$ which tends to infinity; we always assume $x$ is sufficiently large. Implied constants may depend on the series in \eqref{eq:series we prove is irrational} and the associated integer $t$. If an implied constant depends on some other quantity, we occasionally denote this in a subscript, as in $f \ll_A g$. We shall occasionally use without comment the fact that $\omega(n) \ll \log n$. We also use the fact that $\omega$ is additive: given coprime positive integers $m$ and $n$, we have $\omega(mn) = \omega(m) + \omega(n)$. For $\theta\in \mathbb{R}$, we let $\lfloor \theta \rfloor, \lceil \theta\rceil$ denote, respectively, the floor and ceiling of $\theta$. We usually write the congruence condition $n \equiv a \pmod{q}$ as $n\equiv a (q)$. We write $\mathbf{1}_{A(n)}$ for the indicator function of a condition $A(n)$, so that $\mathbf{1}_{A(n)}=1$ if $A(n)$ is true, otherwise it is zero. We write either $\text{gcd}(m,n)$ or $(m,n)$ for the greatest common divisor of two integers (not both zero).

\section{Proof of Theorem \ref{thm:main theorem}}

In this section, we reduce the proof of Theorem \ref{thm:main theorem} to a technical result (Proposition \ref{prop:main prop} below) that provides for the existence of a positive integer with certain desirable properties. In order to state the proposition, we introduce some parameters that appear throughout the paper.

For large positive $x$, define positive integers $K,L,Q$ by
\begin{align}\label{eq:defn of KLQ}
K &= \lfloor 5\log \log \log x\rfloor, \ \ \ \ \ \ L = \lfloor 2 \log \log x \rfloor, \ \ \ \ \ \ Q = \prod_{p \leq K} p^{2\left\lceil \frac{\log K}{\log p}\right\rceil}.
\end{align}
We note that $k^2 \mid Q$ for every positive integer $k\leq K$, and that
\begin{align}\label{eq:size of Q}
(\log \log x)^{10-o(1)} \leq Q \leq (\log \log x)^{20+o(1)}
\end{align}
by the prime number theorem.

\begin{proposition}\label{prop:main prop}
Assume Conjecture \ref{conj:unif HL conj}. Let $x$ be large, and define $K,L,Q$ as in \eqref{eq:defn of KLQ}. Then there exists a positive integer $n_0 \leq x$ such that the following hold:
\begin{enumerate}
\item\label{item1} $n_0\tfrac{Q}{k}+1$ is prime for every $1\leq k \leq K$,
\item\label{item2} $\omega(n_0Q + k) \leq (\log \log x)^2$ for $K < k \leq L$,
\item\label{item3} $\omega(n_0Q+K+1) > \tfrac{1}{10} \log \log x$.
\end{enumerate}
\end{proposition}

\begin{proof}[Proof of Theorem \ref{thm:main theorem} assuming Proposition \ref{prop:main prop}]
Given an integer $t \geq 2$, let
\begin{align*}
\alpha = \alpha_t = \sum_{n\geq 1} \frac{\omega(n)}{t^n}
\end{align*}
denote the series in \eqref{eq:series we prove is irrational}. Assume for contradiction that $\alpha \in \mathbb{Q}$, so that there exist fixed positive integers $a,b$ such that $\alpha = a/b$. Note that, for any positive integer $N$,
\begin{align*}
T(N) \coloneq b\sum_{k\geq 1} \frac{\omega(N+k)}{t^k}= bt^N \left(\alpha - \sum_{n \leq N} \frac{\omega(n)}{t^n} \right)
\end{align*}
is an integer.

Let $x$ be large, and assume Conjecture \ref{conj:unif HL conj}. We set $N=n_0 Q$, where $n_0$ is the positive integer in Proposition \ref{prop:main prop}. We then split $T(n_0Q) = S_1+S_2+S_3$, where
\begin{align*}
S_1 &= b\sum_{k \leq K} \frac{\omega(n_0Q+k)}{t^k}, \\ 
S_2 &= b\sum_{K < k \leq L} \frac{\omega(n_0Q+k)}{t^k}, \\
S_3 &= b\sum_{k > L} \frac{\omega(n_0Q+k)}{t^k} = O\Big(\frac{\log x}{t^L} \Big).
\end{align*}
Since $n_0Q+k = k(n_0 \tfrac{Q}{k}+1)$ and $\text{gcd}(k,n_0 \tfrac{Q}{k}+1)=1$, we see by part \eqref{item1} of Proposition \ref{prop:main prop} that
\begin{align*}
S_1 &= b \sum_{k\leq K} \frac{\omega(k)}{t^k} + b\sum_{k\leq K} \frac{\omega(n_0 \tfrac{Q}{k}+1)}{t^k} = a - b\sum_{k>K}\frac{\omega(k)}{t^k} + b\sum_{k\leq K} \frac{1}{t^k} \\
&= a + \frac{b}{t-1} + O \left(\frac{\log K}{t^K} \right).
\end{align*}
Parts \eqref{item2} and \eqref{item3} of Proposition \ref{prop:main prop} imply
\begin{align*}
\frac{b\log \log x}{10 t^{K+1}} \leq S_2 \leq \frac{bL(\log \log x)^2}{t^K}.
\end{align*}
Hence $T(n_0Q)=a + \frac{b}{t-1} + S_2 + E$ is an integer, where $E$ is some real number with $|E| \ll \tfrac{\log K}{t^K}$. If $b$ is not divisible by $t-1$, then we obtain a contradiction since $S_2=o(1),E = o(1)$ and $b/(t-1)$ is not an integer. If $b$ is divisible by $t-1$ then we also obtain a contradiction, for then $S_2+E$ is nonzero by the lower bound on $S_2$, but $S_2+E = o(1)$.
\end{proof}

We derive Proposition \ref{prop:main prop} from a technical counting result.

\begin{lemma}\label{lem:main lemma}
Let $x$ be large, and let $K,Q$ be defined as in \eqref{eq:defn of KLQ}. For $1\leq k \leq K$, define the linear form $L_k(n) = n\tfrac{Q}{k}+1$. Let
\begin{align*}
\mathfrak{S} \coloneq \prod_{p\leq K}\left(1 - \frac{1}{p} \right)^{-K} \prod_{p>K}\left(1 - \frac{K}{p} \Big)\Big(1 - \frac{1}{p} \right)^{-K}.
\end{align*}
 Then
\begin{align*}
\sum_{\substack{n\leq x \\ L_k(n) \textup{ is prime for }1\leq k \leq K \\ \omega(nQ+K+1) \leq \tfrac{1}{10} \log \log x}} \!\!\!\!\!\!\!\!\!\!\!\!\!\!\!\!\! 1 \, \, \, \, \,\,\,\,\,\,\,\, \leq \mathfrak{S}\frac{x}{(\log x)^K} (\log x)^{-c_0 + o(1)},
\end{align*}
where $c_0 = \tfrac{9-\log 10}{10} = 0.6697\ldots$.
\end{lemma}

\begin{proof}[Proof of Proposition \ref{prop:main prop} assuming Lemma \ref{lem:main lemma}]
The set of linear forms $\mathcal{L}=\{L_k(n)\}_{k=1}^K$ is admissible. Since $\omega_\mathcal{L}(p) = 0$ for $p \leq K$, and $\omega_\mathcal{L}(p) = K$ for $p>K$, we find $\mathfrak{S}(\mathcal{L}) = \mathfrak{S}$. 

In order to prove Proposition \ref{prop:main prop}, it suffices to prove that
\begin{align*}
\mathcal{S} \coloneq \sum_{\substack{n\leq x \\ L_k(n) \textup{ is prime for }1\leq k \leq K \\ \omega(nQ+k) \leq (\log \log x)^2 \text{ for } K < k \leq L \\ \omega(nQ+K+1) > \tfrac{1}{10} \log \log x}} \!\!\!\!\!\!\!\!\!\!\!\!\!\!\!\!\!\!\!\!\!\!\!\! 1
\end{align*}
is positive. By inclusion-exclusion,
\begin{align*}
\mathcal{S} \geq \sum_{\substack{n\leq x \\ L_k(n) \textup{ is prime for }1\leq k \leq K}} \!\!\!\!\!\!\!\!\!\!\!\!\!\!\!\! 1 \, \, \, \, \, \, \, \, \, \, \, - \sum_{\substack{n\leq x \\ L_k(n) \textup{ is prime for }1\leq k \leq K \\ \omega(nQ+K+1) \leq \tfrac{1}{10} \log \log x}} \!\!\!\!\!\!\!\!\!\!\!\!\!\!\!\!\!\! 1 \, \, \, \, \, \, \, \, \, \, \, \, \, \,- \, \,  \mathcal{E},
\end{align*}
where
\begin{align*}
\mathcal{E}=\sum_{\substack{n\leq x \\ \exists k \in (K,L], \, \omega(nQ + k) > (\log \log x)^2}} \!\!\!\!\!\!\!\!\!\!\!\!\!\!\!\!\!\!\!\! 1.
\end{align*}
By \eqref{eq:defn of KLQ}, \eqref{eq:size of Q}, Conjecture \ref{conj:unif HL conj}, and Lemma \ref{lem:main lemma}, it therefore suffices to show that
\begin{align}\label{eq:desired bound for mathcal E}
\mathcal{E}=o\Big(\mathfrak{S} \frac{x}{(\log x)^K} \Big).
\end{align}
Since $\tau(n) \geq 2^{\omega(n)}$,
\begin{align}\label{eq:upp bound on mathcal E}
\mathcal{E} &\ll L\sum_{m\ll xQ} \tau(m) 2^{-(\log \log x)^2} \ll xLQ (\log x)2^{-(\log \log x)^2}.
\end{align}
By crude estimation, we see that \eqref{eq:upp bound on mathcal E} implies \eqref{eq:desired bound for mathcal E} as long as, say,
\begin{align}\label{eq:low bound on sing series}
\mathfrak{S} \geq (\log x)^{-1}.
\end{align}

We now bound $\mathfrak{S}$ from below. The portion of the product with $p\leq K$ is clearly $\geq 1$, and we split the remaining portion of the product into $K < p \leq 2K$ and $p>2K$. Since $(1-\tfrac{K}{p})(1-\tfrac{1}{p})^{-K}$ is increasing in $p$,
\begin{align*}
\prod_{K<p \leq 2K}\left(1 - \frac{K}{p} \right)\left(1 - \frac{1}{p} \right)^{-K} \geq \prod_{K<p\leq 2K}\left(1 - \frac{K}{K+1} \right)\left(1 - \frac{1}{K+1} \right)^{-K}\geq K^{-K}.
\end{align*}
For $p>2K$, we have
\begin{align*}
\prod_{p> 2K}\left(1 - \frac{K}{p} \right)\left(1 - \frac{1}{p} \right)^{-K} &\geq \exp \left(-\sum_{p>2K}\sum_{\ell \geq 2}\frac{1}{\ell}\Big(\frac{K}{p} \Big)^\ell \right)\geq \exp\left(-K^2\sum_{p>2K}\frac{1}{p^2} \right) \geq e^{-K}.
\end{align*}
Therefore $\mathfrak{S} \geq K^{-2K}$, and \eqref{eq:low bound on sing series} follows, with plenty of room to spare, by \eqref{eq:defn of KLQ}.
\end{proof}

The remainder of the paper is devoted to proving Lemma \ref{lem:main lemma}.

\section{Proof of Lemma \ref{lem:main lemma}}

Before beginning the proof of Lemma \ref{lem:main lemma}, we briefly describe the main ideas. First, by using sieves, we may win a factor of $\mathfrak{S}(\log x)^{-K+o(1)}$ from the fact that $L_k(n)$ is prime for $1 \leq k \leq K$. Second, since integers $n \asymp x$ typically have $\omega(n) = (1+o(1))\log \log x$, the condition that $\omega(nQ+K+1) \leq \frac{1}{10} \log \log x$ is atypical. Indeed, we expect to save a factor of $(\log x)^{-c}$, for some constant $c>0$. If these two sources of savings are ``independent,'' then we obtain the desired result.

Of course, we encounter some technical difficults in executing this strategy. Our method for controlling the condition $\omega(nQ+K+1) \leq \frac{1}{10} \log \log x$, when combined with a sieve, requires us to consider integrals involving Dirichlet $L$-functions to fractional exponents. To maintain holomorphy, we only shift contours of integration in a region devoid of zeros of these $L$-functions. Thus, obtaining good bounds on these integrals requires the use of zero-free regions and zero-density estimates. Since $K$ tends to infinity, it is necessary to account for possible Siegel zeros in order to obtain sufficiently strong error terms.

We begin the proof. By summing dyadically, it suffices to show, for large $x$, that
\begin{align*}
\mathcal{Z} &\coloneq \sum_{\substack{x/2<n\leq x \\ L_k(n) \textup{ is prime for }1\leq k \leq K \\ \omega(nQ+K+1) \leq \tfrac{1}{10} \log \log x}} \!\!\!\!\!\!\!\!\!\!\!\!\!\! 1 \,\,\,\,\,\,\,\leq \mathfrak{S}\frac{x}{(\log x)^K} (\log x)^{-c_0 + o(1)}.
\end{align*}
For $\lambda \in (0,1)$, where $\lambda$ is a constant we choose later, we find
\begin{align*}
\mathbf{1}_{\omega(nQ+K+1) \leq \tfrac{1}{10} \log \log x} \leq \lambda^{\omega(nQ+K+1)-\tfrac{1}{10}\log \log x}.
\end{align*}
We change variables, writing $n$ for $nQ+K+1$, so that
\begin{align*}
\mathcal{Z}&\leq (\log x)^{\tfrac{1}{10}\log(1/\lambda)}\sum_{\substack{xQ/2 < n \leq 2xQ \\ n\equiv K+1 (Q) \\ \tfrac{n-K-1}{k}+1 \textup{ is prime for }1\leq k \leq K}}\!\!\!\!\!\!\!\!\!\!\!\!\!\! \lambda^{\omega(n)}.
\end{align*}
Let $g = \text{gcd}(K+1,Q)$, and observe that $g \mid n$. We define $Q'=Q/g,K'=(K+1)/g$, and change variables $n \rightarrow gn$ to obtain
\begin{align*}
\mathcal{Z} &\leq (\log x)^{\tfrac{1}{10}\log(1/\lambda)}\sum_{\substack{xQ'/2 < n \leq 2xQ' \\ n\equiv K' (Q') \\ \tfrac{ng-K-1}{k}+1 \textup{ is prime for }1\leq k \leq K}}\!\!\!\!\!\!\!\!\!\!\!\!\!\! \lambda^{\omega(gn)}.
\end{align*}
We note that $\text{gcd}(K',Q')=1$. Since $\omega(gn) \geq \omega(n)$, we can replace $\lambda^{\omega(gn)}$ by $\lambda^{\omega(n)}$ for an upper bound.

We introduce a smoothing for technical convenience. We encapsulate the desired properties of the smooth function in the following lemma.

\begin{lemma}\label{lem:smooth fn W}
There is a nonnegative smooth function $W(x)$ which is compactly supported in $[1/4,4]$, and which is equal to one in $[1/2,2]$. Furthermore, the derivatives of $W$ satisfy $|W^{(j)}(x)| \ll j^{3j}$, where the implied constant is absolute. If we write $W^\dagger(s)$ for the Mellin transform of $W$, then $W^\dagger(s)$ is entire, and there exists an absolute constant $c>0$ such that
\begin{align*}
|W^\dagger(s)| &\ll 4^{|\textup{Re}(s)|} \exp(-c|s|^{1/3}).
\end{align*}
\end{lemma}
\begin{proof}
The construction in \cite[appendix A]{Iw2014} yields a smooth function $W(x) = G(x-\frac{1}{2})$ which is nonnegative, supported in $[1/4,4]$, and equal to one on $[1/2,2]$ (take $T=\frac{1}{2}, U = \frac{3}{2}, V = \frac{1}{4}$). By differentiation and \cite[Corollary A.2]{Iw2014}, we have $|W^{(j)}(x)| \ll 4^j j! (2j/e)^j\ll j^{3j}$, the implied constants being absolute. Since $W$ is compactly supported away from zero, the Mellin transform
\begin{align*}
W^\dagger(s) = \int_0^\infty W(x) x^{s-1} dx
\end{align*}
is entire. To prove the bound on $|W^\dagger(s)|$, we may assume that $|s|$ is sufficiently large, otherwise the bound is trivial by suitably adjusting the implied constant. For any positive integer $k$, integration by parts yields
\begin{align*}
W^\dagger(s) = \frac{(-1)^k}{s(s+1)\cdots(s+k-1)}\int_0^\infty W^{(k)}(x) x^{s+k-1}dx.
\end{align*}
By the triangle inequality and the bound on derivatives of $W$, we obtain
\begin{align*}
|W^\dagger(s)| &\ll 4^{|\text{Re}(s)|}4^k k^{3k} \prod_{j=0}^{k-1} |s+j|^{-1}.
\end{align*}
We conclude by choosing $k = \lfloor \delta |s|^{1/3}\rfloor$, where $\delta>0$ is a sufficiently small constant.
\end{proof}

Taking $W$ to be the smooth function in Lemma \ref{lem:smooth fn W}, we obtain
\begin{align*}
\mathcal{Z} &\leq (\log x)^{\tfrac{1}{10}\log(1/\lambda)}\sum_{\substack{n\equiv K' (Q') \\ \tfrac{ng-K-1}{k}+1 \textup{ is prime for }1\leq k \leq K}}\!\!\!\!\!\!\!\!\!\!\!\!\!\! \lambda^{\omega(n)} W\left(\frac{n}{xQ'}\right).
\end{align*}

We now account for potential Siegel zeros, following the strategy of \cite{May2016}. For a suitable positive constant $c$, there is at most one real primitive Dirichlet character $\chi_*$ to a modulus $q_* \leq \exp((\log x)^{1/3})$ for which $L(s,\chi_*)$ has a real zero $\beta$ which is greater than $1 - \frac{c}{(\log x)^{1/3}}$ \cite[p. 95]{Dav2000}. If $q_*$ exists, we let $B$ be the largest prime factor of $q_*$. By the class number formula, we have $q_* \geq (\log x)^{1/2}$, say (see \cite[p. 96, equation (12)]{Dav2000}). As $q_*$ is squarefree apart from a bounded power of two, it follows that $\log \log x \ll B \leq \exp((\log x)^{1/3})$. If no such $q_*$ exists, we set $B=1$. 

We define $\mathcal{L}(n) \coloneq \prod_{k=1}^K (\tfrac{ng-K-1}{k}+1)$, $X \coloneq x^{1/(\log \log x)^3}$, $\mathcal{P}\coloneq\prod_{K^2 < p \leq X} p$, and $V \coloneq 2\lfloor (\log \log x)^2 \rfloor$. Since $B \leq X$, we use the Brun sieve (e.g. \cite[p. 56, equation (6.6)]{FI2010}) to obtain
\begin{align*}
\mathcal{Z}&\leq (\log x)^{\tfrac{1}{10}\log(1/\lambda)}\sum_{\substack{n\equiv K' (Q') \\ \text{gcd}(\mathcal{L}(n),\mathcal{P}/B)=1}} \lambda^{\omega(n)} W\left(\frac{n}{xQ'}\right) \\
&\leq(\log x)^{\tfrac{1}{10}\log(1/\lambda)}\sum_{\substack{n\equiv K' (Q')}} \lambda^{\omega(n)}W\left(\frac{n}{xQ'}\right)\sum_{\substack{d \mid \mathcal{L}(n) \\ d \mid \mathcal{P} \\ (d,B)=1 \\ \omega(d) \leq V}}\mu(d).
\end{align*}
Swapping the order of summation then gives
\begin{align*}
\mathcal{Z} &\leq (\log x)^{\tfrac{1}{10}\log(1/\lambda)}\sum_{\substack{d \mid \mathcal{P} \\  (d,B)=1 \\ \omega(d) \leq V}} \mu(d) \sum_{\substack{n\equiv K' (Q') \\ d \mid \mathcal{L}(n)}} \lambda^{\omega(n)}W\left(\frac{n}{xQ'}\right).
\end{align*}
Since $d$ is squarefree and has no prime divisors $\leq K^2$, we see by the Chinese remainder theorem that $d\mid \mathcal{L}(n)$ if and only if $n$ lies in a particular set of reduced residue classes $S(d)$ modulo $d$ of size $K^{\omega(d)}$. We therefore have $\mathcal{Z} \leq (\log x)^{\tfrac{1}{10}\log(1/\lambda)} (\mathcal{M} + \mathcal{R})$, where
\begin{align}
\mathcal{M} &\coloneq \sum_{\substack{d \mid \mathcal{P} \\  (d,B)=1 \\ \omega(d) \leq V}} \mu(d)\frac{K^{\omega(d)}}{\varphi(Q'd)}\sum_{\substack{(n,Q'd)=1}} \lambda^{\omega(n)}W\left(\frac{n}{xQ'}\right)
\end{align}
and
\begin{align}
\begin{split}
\mathcal{R} &\coloneq \sum_{\substack{d \mid \mathcal{P} \\  (d,B)=1 \\ \omega(d) \leq V}} \mu(d) \sum_{v \in S(d)}\Bigg\{\sum_{\substack{n\equiv K' (Q') \\ n\equiv v (d)}} \lambda^{\omega(n)}W\left(\frac{n}{xQ'}\right) - \frac{1}{\varphi(Q'd)}\sum_{\substack{(n,Q'd)=1}} \lambda^{\omega(n)}W\left(\frac{n}{xQ'}\right) \Bigg\}.
\end{split}
\end{align}
We note that $Q'$ and $d$ are coprime, so by the Chinese remainder theorem we may combine the congruence conditions $n\equiv K'(Q')$ and $n \equiv v (d)$ into a single congruence condition modulo $Q'd$.

We first bound the error term $\mathcal{R}$, and then we turn to the main term $\mathcal{M}$.

\subsection{Bounding the error term $\mathcal{R}$}

We apply the triangle inequality and take the worst residue class in $S(d)$ to get
\begin{align*}
|\mathcal{R}| &\leq \sum_{\substack{m \leq Q'X^V \\ (m,B)=1}}K^{\omega(m)} E(m),
\end{align*}
where
\begin{align*}
E(m) \coloneq \max_{(\gamma,m)=1}\Bigg|\sum_{\substack{n\equiv \gamma (m)}} \lambda^{\omega(n)}W\Big(\frac{n}{xQ'}\Big) - \frac{1}{\varphi(m)}\sum_{\substack{(n,m)=1}} \lambda^{\omega(n)}W\left(\frac{n}{xQ'}\right) \Bigg|.
\end{align*}
By Cauchy-Schwarz and the trivial bound $E(m) \ll xQ'/m$, we deduce
\begin{align}\label{eq:cauchy bound for mathcal R}
|\mathcal{R}|^2 &\ll x (\log x)^{3K^2} \mathcal{R}_1,
\end{align}
where
\begin{align*}
\mathcal{R}_1 &= \sum_{\substack{m\leq Q'X^V \\ (m,B)=1}} E(m).
\end{align*}
We apply multiplicative characters to detect the congruence in $E(m)$ and use the triangle inequality to obtain
\begin{align*}
\mathcal{R}_1 &\leq \sum_{\substack{m\leq Q'X^V \\ (m,B)=1}} \frac{1}{\varphi(m)}\sum_{\substack{\chi (m) \\ \chi\neq \chi_0}}\Bigg|\sum_{n}\lambda^{\omega(n)}\chi(n)W\left(\frac{n}{xQ'}\right)  \Bigg|,
\end{align*}
where $\chi_0$ denotes the principal character. We then replace each character $\chi$ modulo $m$ by the primitive character $\psi$ modulo $r$ which induces it. After splitting the range of $r$ dyadically, we obtain
\begin{align*}
\mathcal{R}_1 &\ll (\log x) \sup_{1 \ll R \ll Q'X^V} \, \sum_{\ell \leq Q'X^V}\frac{1}{\varphi(\ell)}\sum_{\substack{r \asymp R \\ (r,B)=1}} \frac{1}{\varphi(r)}\sideset{}{^*}\sum_{\psi (r)}\Bigg|\sum_{(n,\ell)=1}\lambda^{\omega(n)}\psi(n)W\left(\frac{n}{xQ'}\right) \Bigg|,
\end{align*}
where the starred sum denotes a sum over primitive characters.

We use the following lemma to bound $\mathcal{R}_1$.

\begin{lemma}\label{lem:bound lam psi sum via rectangle}
Let $x$ be large, and let $\mathcal{C}\subseteq \mathbb{C}$ denote the region
\begin{align*}
\mathcal{C} = \{u+iv : u\geq \tfrac{1}{2}, |v| \leq (\log x)^{20}\}.
\end{align*}
For a primitive Dirichlet character $\psi \pmod{r}$ with $r\leq x^{1/3}$, let $\sigma(\psi)$ denote the largest real part of a zero of $L(s,\psi)$ contained in $\mathcal{C}$. If $\lambda \in (0,\tfrac{1}{2}]$, $x^{1/2}\leq N \leq x^2$, and $\ell \leq x^{O(1)}$, then
\begin{align*}
\Bigg|\sum_{(n,\ell)=1}\lambda^{\omega(n)}\psi(n)W\left(\frac{n}{N}\right)  \Bigg| \ll r^{1/2} N^{\max(\tfrac{9}{10},\sigma(\psi))} \prod_{p \mid \ell} \left(1 + \frac{3}{p^{9/10}} \right).
\end{align*}
\end{lemma}

We show how Lemma \ref{lem:bound lam psi sum via rectangle} allows us to bound $\mathcal{R}_1$ before proceeding to its proof. We treat two different cases, depending on whether $R \leq \frac{1}{10} \exp((\log x)^{1/3})$.

If $R \leq \frac{1}{10} \exp((\log x)^{1/3})$, then since $(r,B)=1$ the classical zero-free region (see \cite[Chapter 14]{Dav2000}) implies $\sigma(\psi) \leq 1 - \frac{c}{(\log x)^{1/3}}$ for every $\psi$, for some constant $c>0$. It follows from Lemma \ref{lem:bound lam psi sum via rectangle} that
\begin{align}\label{eq:R1 bound for small R}
\mathcal{R}_1 &\ll(\log x)^2 R^{3/2} (xQ')^{1-c(\log x)^{-1/3}} \ll x \exp(-\sqrt{\log x}).
\end{align}

Assume now that $R\gg\exp((\log x)^{1/3})$. We break up the sum over $\psi$ according to the value of $\sigma(\psi)$. The contribution from those $\psi$ with $\sigma(\psi) \leq \tfrac{9}{10}$ is $\leq x^{9/10+o(1)}$. For those $\psi$ with $\sigma(\psi) > \tfrac{9}{10}$, the contribution is
\begin{align}\label{eq:intermed R 1 sum}
\ll \frac{(\log x)^{O(1)}}{R^{1/2}} \sum_{\substack{\tfrac{9}{10} \leq \beta \leq 1 \\ \beta = \tfrac{9}{10} + \tfrac{j}{\log x} \\ j \geq 0}} x^\beta \sum_{r\asymp R} \, \sideset{}{^*}\sum_{\substack{\psi(r) \\ \beta \leq \sigma(\psi) < \beta + \tfrac{1}{\log x}}} 1.
\end{align}
Let $N(\sigma,T,\psi)$ denote the number of zeros $\rho$ of $L(s,\psi)$ satisfying $\text{Re}(\rho) \geq \beta$ and $|\text{Im}(\rho)| \leq T$. If $\beta \leq \sigma(\psi)$, then $N(\beta,(\log x)^{20},\psi) \geq 1$, so \eqref{eq:intermed R 1 sum} is
\begin{align*}
&\ll \frac{(\log x)^{O(1)}}{R^{1/2}} \sup_{\tfrac{9}{10}\leq \beta \leq 1} x^\beta \sum_{r\asymp R}\sideset{}{^*}\sum_{\psi(r)} N(\beta,(\log x)^{20},\psi).
\end{align*}
By work of Montgomery \cite[Theorem 12.2]{Mont1971} we have
\begin{align*}
\sum_{r\asymp R}\sideset{}{^*}\sum_{\psi(r)} N(\beta,(\log x)^{20},\psi) &\ll R^{5(1-\beta)}(\log x)^{O(1)}
\end{align*}
since $\beta \geq \frac{9}{10}$, and therefore
\begin{align}\label{eq:bound for large R}
\mathcal{R}_1 &\ll x^{9/10+o(1)}+ (\log x)^{O(1)}\frac{x}{R^{1/2}} \sup_{\tfrac{9}{10}\leq \beta \leq 1}\Big(\frac{x}{R^5}\Big)^{\beta-1} \ll \frac{x}{R^{1/3}}.
\end{align}
Taking \eqref{eq:R1 bound for small R} and \eqref{eq:bound for large R} together shows $|\mathcal{R}_1| \ll x \exp(-c(\log x)^{1/3})$, for some positive constant $c$. Inserting this bound in \eqref{eq:cauchy bound for mathcal R} then gives
\begin{align}\label{eq:final bound on mathcal R}
|\mathcal{R}| \ll x \exp(-(\log x)^{1/4}).
\end{align}

\begin{proof}[Proof of Lemma \ref{lem:bound lam psi sum via rectangle}]
Let $W^\dagger(s)$ denote the Mellin transform of $W$. By Mellin inversion,
\begin{align*}
\sum_{(n,\ell)=1}\lambda^{\omega(n)}\psi(n)W\left(\frac{n}{N}\right) = \frac{1}{2\pi i}\int_{c-i\infty}^{c+i\infty}N^s W^\dagger(s) \prod_{p\mid \ell}\left(1+\lambda\frac{\psi(p)}{p^s-\psi(p)}\right)^{-1}A(s) L(s,\psi)^\lambda ds,
\end{align*}
where $c>1$, and $A(s)$ is an Euler product (depending on $\lambda$ and $\psi$) that converges absolutely for $\text{Re}(s) \geq \tfrac{9}{10}$, and which is uniformly bounded in this region. By Lemma \ref{lem:smooth fn W}, we may truncate the integral to $|s| \leq (\log x)^{10}$ at the cost of negligible error. By the definition of $\sigma(\psi)$, we see $L(s,\psi)^\lambda$ is holomorphic for $\text{Re}(s) > \sigma(\psi)$ and $|\text{Im}(s)| \leq (\log x)^{20}$, so we may move the line of integration from $c$ to $\beta = \max(\tfrac{9}{10},\sigma(\psi)) + \tfrac{1}{\log x}$. We apply the triangle inequality to the integrals over this new vertical line segment and the connecting horizontal line segments. Since $0 < \lambda \leq \frac{1}{2}$, the easy bound \cite[Lemma 10.15]{MV2007} implies
\begin{align*}
|L(\sigma+it,\psi)^\lambda| &\ll 1+ |L(\sigma+it,\psi)|^{1/2} \ll r^{\frac{1-\sigma}{2}+\varepsilon}(1+|t|) \ll r^{1/2}(1+|t|),
\end{align*}
say. Also, for $\text{Re}(s) \geq \tfrac{9}{10}$,
\begin{align*}
\Bigg|\prod_{p\mid \ell}\left(1+\lambda\frac{\psi(p)}{p^s-\psi(p)}\right)^{-1} \Bigg| &\leq \prod_{p\mid \ell}\left(1 + \frac{\lambda}{p^{9/10}-1-\lambda}\right)\leq \prod_{p \mid \ell} \left(1 + \frac{3}{p^{9/10}} \right),
\end{align*}
since $\lambda\leq 1/2$.
\end{proof}

\subsection{Evaluating the main term $\mathcal{M}$}

We write
\begin{align*}
\mathcal{M} &= \frac{1}{\varphi(Q')}\sum_{(n,Q')=1}\lambda^{\omega(n)}W\left(\frac{n}{xQ'}\right) \sum_{\substack{d \mid \mathcal{P} \\ (d,Bn)=1 \\ \omega(d) \leq V}}\frac{\mu(d)}{\varphi(d)}K^{\omega(d)}.
\end{align*}
By \cite[p. 56, equation (6.8)]{FI2010}, we see
\begin{align*}
\sum_{\substack{d \mid \mathcal{P} \\ (d,Bn)=1 \\ \omega(d) \leq V}}\frac{\mu(d)}{\varphi(d)}K^{\omega(d)} &= \sum_{\substack{d \mid \mathcal{P} \\ (d,Bn)=1}}\frac{\mu(d)}{\varphi(d)}K^{\omega(d)} + O \Bigg(\sum_{\substack{d \mid \mathcal{P}\\ \omega(d) = V+1}}\frac{K^{\omega(d)}}{\varphi(d)} \Bigg),
\end{align*}
and we bound the error by
\begin{align*}
\sum_{\substack{d \mid \mathcal{P}\\ \omega(d) = V+1}}\frac{K^{\omega(d)}}{\varphi(d)}\leq \frac{1}{(V+1)!}\left(\sum_{K^2 < p \leq X} \frac{K}{p-1} \right)^{V+1} \ll \exp(-(\log \log x)^2).
\end{align*}
We arrange the complete sum over $d$ as
\begin{align*}
\sum_{\substack{d \mid \mathcal{P} \\ (d,Bn)=1}}\frac{\mu(d)}{\varphi(d)}K^{\omega(d)} &= \prod_{\substack{K^2 < p \leq X \\ p \nmid Bn}}\left(1 - \frac{K}{p-1}\right) = (1+o(1))\mathfrak{S}\prod_{p \leq X}\left(1 - \frac{1}{p}\right)^K \\
&\times \prod_{\substack{p \mid Bn \\ K^2 < p \leq X}} \left( 1 - \frac{K}{p-1}\right)^{-1}\prod_{K < p \leq K^2}\left(1 - \frac{K}{p}\right)^{-1}.
\end{align*}
By Mertens' theorem and some estimations, this is
\begin{align*}
&\leq \mathfrak{S}\frac{(\log \log x)^{O(K)}}{(\log x)^K}e^{O(K)}\prod_{p \mid Bn}(1 + p^{-1/3} ).
\end{align*}
It follows that
\begin{align*}
\mathcal{M} &\leq \mathfrak{S}\frac{(\log x)^{o(1)}}{(\log x)^K}\frac{1}{\varphi(Q')}\sum_{n \asymp xQ'}\lambda^{\omega(n)} \prod_{p \mid n}(1 + p^{-1/3}).
\end{align*}
Work of Shiu \cite[Theorem 1]{Shiu1980} implies
\begin{align*}
\sum_{n \asymp xQ'}\lambda^{\omega(n)} \prod_{p \mid n}(1 + p^{-1/3})\ll \frac{xQ'}{\log x}\exp\left(\sum_{p\ll xQ'}\frac{\lambda}{p}(1+p^{-1/3}) \right) \ll \frac{xQ'}{(\log x)^{1-\lambda}},
\end{align*}
and it follows that
\begin{align}\label{eq:final bound on mathcal M}
\mathcal{M}&\leq \mathfrak{S}\frac{x}{(\log x)^K} (\log x)^{-1+\lambda + o(1)}.
\end{align}
Recalling from \eqref{eq:low bound on sing series} that $\mathfrak{S} \geq (\log x)^{-1}$, comparing \eqref{eq:final bound on mathcal M} and \eqref{eq:final bound on mathcal R} yields
\begin{align*}
\mathcal{Z} &\leq \mathfrak{S}\frac{x}{(\log x)^K} (\log x)^{-1+\lambda+\frac{1}{10}\log(1/\lambda) + o(1)}.
\end{align*}
We choose the optimal $\lambda = \tfrac{1}{10}$ (which is suitable for Lemma \ref{lem:bound lam psi sum via rectangle}) to conclude.

\section*{Acknowledgments}

The author is supported by National Science Foundation grant DMS-2418328.

\bibliographystyle{plain}
\bibliography{refs}

\end{document}